\documentclass[reqno]{amsart}

\usepackage{amsmath,amssymb,amscd, accents} \usepackage{graphicx}
\DeclareGraphicsExtensions{.eps} \usepackage{mathrsfs}
\usepackage[mathcal]{eucal}

\newtheorem{Thm}{Theorem}{\bfseries}{\itshape}
\newtheorem*{Thm*}{Theorem}{\bfseries}{\itshape}
\newtheorem{Cor}{Corollary}{\bfseries}{\itshape}
\newtheorem{Prop}[Cor]{Proposition}{\bfseries}{\itshape}
\newtheorem{Lem}[Cor]{Lemma}{\bfseries}{\itshape}
\newtheorem*{Lem*}{Lemma}{\bfseries}{\itshape} 
{\bfseries}{\itshape}
\newtheorem{Conj}[Cor]{Conjecture}{\bfseries}{\itshape}
\newtheorem{Def}[Cor]{Definition}{\bfseries}{\rmfamily}
{\scshape}{\rmfamily}
\newtheorem{Rem}[Cor]{Remark}{\scshape}{\rmfamily}
{\bfseries}{\itshape}

\renewcommand\ge{\geqslant} \renewcommand\le{\leqslant}
\let\tildeaccent=\~ \let\hataccent=\^
\renewcommand\~[1]{\widetilde{#1}}

\def\<{\left<} \def\>{\right>} \def\({\left(} \def\){\right)}

 \def\norm#1{\left\Vert #1
  \right\Vert}

 \def\pd#1#2{\tfrac{\partial#1}{\partial#2}}

\let\polishL=l \def\Zoladek.{\.Zol\c adek}

 \def\const{\operatorname{const}}

 \def\dist{\operatorname{dist}}

\def\etc.{\emph{etc}.}

\def\:{\colon} \def\R{{\mathbb R}} \def\C{{\mathbb C}}  \def\N{{\mathbb N}} \def\Q{{\mathbb Q}}

\def\A{{\mathbb A}}
\def\K{{\mathbb K}}
\def\M{{\mathbb M}}
\def\B{{\mathbb B}}

\def\K{{\mathbb K}}

\let\PolishL=\L 
\def\L{{\mathbb L}}

 \def\e{\varepsilon} 
 \let\ol=\overline  
 
\def\poly{\operatorname{poly}}

 \def\d{\,\mathrm d}

 \def\Lojas.{\PolishL ojasiewicz}

\def\cF{{\mathcal F}} \def\cL{{\mathcal L}}

\def\cB{{\mathcal B}}

\def\cF{{\mathcal F}}

\def\cO{{\mathcal O}}
\def\cA{{\mathcal A}}

\def\rest#1{{\vert_{#1}}}

\def\Gal{\operatorname{Gal}}

\def\vxi{{\boldsymbol\xi}}

\def\fg{{\mathfrak g}}

\def\disc{\operatorname{disc}}

\def\bG{\mathbf{G}}
\def\bT{\mathbf{T}}
\def\bX{\mathbf{X}}

\begin{document}

\title[Galois lower bounds for special points]{Lower bounds for
Galois orbits of
  special points on Shimura varieties: a point-counting approach}

\author{Gal Binyamini}
\address{Weizmann Institute of Science, Rehovot, Israel}
\email{gal.binyamini@weizmann.ac.il}

\author{Harry Schmidt}
\address{Universit\"at Basel, Basel, Switzerland}
\email{harry.schmidt@unibas.ch} 

\author{Andrei Yafaev}
\address{University College London, UK}
\email{yafaev@math.ucl.ac.uk}

\dedicatory{To Bas Edixhoven on the occasion of his 60th birthday.}

\thanks{This research was supported by the ISRAEL SCIENCE FOUNDATION
  (grant No. 1167/17) and by funding received from the MINERVA
  Stiftung with the funds from the BMBF of the Federal Republic of
  Germany. This project has received funding from the European
  Research Council (ERC) under the European Union's Horizon 2020
  research and innovation programme (grant agreement No 802107)
 Yafaev was supported by a Leverhulme research grant RPG-2019-180.
  Schmidt was supported by the Engineering and Physical Sciences
Research Council  grant EP/N007956/1.}
	

\begin{abstract}
  Let $S$ be a Shimura variety.
We conjecture that the heights of
  special points in $S(\overline{\Q})$ are discriminant negligible
  with respect to some Weil height function
  $h:S(\overline\Q)\to\R$.
    
  Assuming this
  conjecture to be true, we prove that the sizes of the Galois orbits
  of special points grow as a fixed power of their discriminant (an
  invariant we will define in the text). In particular, we give a new proof of a theorem of Tsimerman on
lower bounds for Galois degrees of special points in Shimura varieties of abelian type. 
This gives a new proof of the Andr\'e-Oort conjecture for such varieties that avoids the use of
Masser-W\"ustholz isogeny estimates, replacing them by a
point-counting argument. 
\end{abstract}
\maketitle

\date{\today}

\section{Introduction}
\label{sec:intro}

For terminology, facts and notations concerning Shimura varieties, we refer to \cite{Milne} and references therein.
Let $(\bG,\bX)$ be a Shimura datum.
We assume that $\bG$ is semisimple of adjoint type. 
This assumption does not cause any loss of generality with regards to the problem of bounding Galois degrees and 
 applications to Andr\'e-Oort type questions: one can always reduce to this situation (see for example Proposition 2.2 of \cite{EY}).

Let $K$ be a compact open subgroup 
of $\bG(\A_f)$ and  $X$ a connected component of $\bX$.
We let $\bG(\Q)^+$ be the group of
$\Q$-points of $\bG$ contained in the neutral component of $\bG(\R)$ and we 
let
$\Gamma = K\cap \bG(\Q)^+$.

This data defines the Shimura variety 
$$
Sh_K(\bG,\bX) = \bG(\Q) \backslash \bX \times \bG(\A_f)  /K
$$
and its distinguished connected component
$S:=\Gamma\backslash X$.
By standard abuse of terminology, we will refer to $S$ as a `Shimura variety'
(even if technically it is only a connected component of a Shimura variety).

The variety $S$ is a quasi-projective
algebraic variety admitting a canonical model over an explicitly described number field $F$
of degree bounded in terms of the data $(\bG,\bX)$ and $K$.

Let $\pi:X\to S$
denote the quotient map by $\Gamma$, and choose $\Lambda\subset X$
a semialgebraic fundamental set for the action of $\Gamma$ on $X$
as in Theorem 3.1 of \cite{KUY}.

Recall that \emph{special} points on $S$ are algebraic points defined over 
abelian extensions of $F$.  Estimating the degrees of these extensions is 
a difficult problem. To a special point one
attaches a quantity which we call its `discriminant' and it is conjectured that
 degrees of special points (over $F$) grow as a power of their discriminant.
At present, this conjecture is only known under the assumption of the 
Generalised Riemann Hypothesis (see \cite{UY1}).
It is also known for all Shimura varieties of abelian type (see \cite{Tsimerman} and \cite{Tsimerman1}).
This result relies on two major ingredients: one- the averaged Colmez formula for Faltings heights
of CM abelian varieties and two - the Masser-W\"ustholtz isogeny theorem.

The Masser-W\"ustholz theorem does not seem to be easily generalisable to the case of general Shimura varieties.
In this paper we replace its usage with a counting theorem due to the first author
which holds for all Shimura varieties.
The question of bounding the height remains a major obstacle.
We formulate a conjecture on height bounds, assuming which, we are able to
obtain the required estimate for the degrees of special points.

To formulate our conjecture we need to introduce some technical notations.
Let $p$ be a special point of $S$ and write $p=\pi(x)$ with $x$ in $\Lambda$.
Let $\bT$ be the Mumford-Tate group of $x$. For a definition of the Mumford-Tate group 
we refer to \cite{Milne}.
By definition of a special point, $\bT$ is an algebraic torus.
Let $K^m_{\bT}$ be the maximal compact open subgroup of $\bT(\A_f)$
($\A_f$ denotes the finite ad\`eles of $\Q$) and $K_{\bT}$ the compact open
subgroup $K\cap \bT(\A_f)$ of $\bT(\A_f)$.
Let $L$ be the splitting field of $\bT$ i.e. the smallest extension of $\Q$ such that
$\bT_{L}$ is a split torus. Under our assumption ($\bG$ is of adjoint type),
$L$ is a Galois CM field. We let $d_L$ be the absolute value of the discriminant of $L$.

\begin{Def}
With the notations introduced above, we define the \emph{discriminant} of $p$
as
$$
\disc(p) = [K^m_{\bT}: K_{\bT}] d_L.
$$
\end{Def}

 We now formulate our conjecture on heights. In what follows we write
 $a = O_b(c)$ (respectively, $a = \poly_b(c)$) to indicate that $a \le \gamma(b)\cdot c$ (respectively,
 $a \le (c+1)^{\gamma(b)}$) where $\gamma(\cdot)$ is some
 \emph{universally fixed} function (which may be different for each
 occurrence of this notation in the text). Here $a$ and $c$ denote
 natural numbers, and $b$ can involve one or several arguments of any
 type.  We also allow several arguments in $\poly_b(c_1,\ldots,c_n)$
 which we interpret as $\poly_b(c_1+\cdots+c_n)$.
 We will also sometimes use notation $\gg_S$ or $\ll_S$ to mean bigger than (resp. smaller than)
 up to a constant depending on $S$ only.

\begin{Conj}[Conjecture on heights of special points on $S$]\label{conj:main}
There exists a Weil height function  $h \colon S(\overline\Q) \to \R$
such that the following holds.

Let $p\in S$ be a special point, then for any $\e>0$ we have
  \begin{equation}
    h(p) = O_{S,\e}(\disc(p)^\e).
  \end{equation}
  
    We then say that the heights of special points are
  \emph{discriminant-negligible} (some authors use the terminology `sub-polynomial in the discriminant').
\end{Conj}

\begin{Rem}
The Shimura variety $S$ admits a Baily-Borel compactification $\overline{S}$
defined over the same field as that of $S$,
and thus there is a a natural Weil height function $h \colon S(\overline\Q) \to \R$.
This is a natural candidate to study. 
The other important compactification is the toroidal one, also yielding a 
Weil height function. From well-known properties of Weil heights (see for example \cite[B.3]{HZ}) follows that if  conjecture \ref{conj:main} holds for one choice of Weil height on $S$ it does hold for every choice.  
\end{Rem}

Our main goal in the present paper is to prove the
following.

\begin{Thm}\label{thm:main}
Assume Conjecture \ref{conj:main} for a Shimura variety $S$. There are $C>0$ and $\epsilon>0$, depending only on $S$ and $F$, such that the following holds.
Let $p$ be a  special point of $S$, then
$$
[F(p):F] > C \disc(p)^{\epsilon}.
$$ 
\end{Thm}

The conclusion of Theorem~\ref{thm:main} for an arbitrary Shimura variety $S$ is the
only missing ingredient in a proof of the Andr\'e-Oort conjecture in full generality
using the Pila-Zannier strategy (see \cite{Tsimerman} and \cite{Gao}). 

We have:

\begin{Thm}\label{thm:AO-S}
The following hold:
\begin{enumerate}
\item
Assume conjecture \ref{conj:main} holds for the Shimura variety $S$.
Then the Andr\'e-Oort conjecture holds for $S$ and any mixed Shimura variety whose pure part is $S$.
\item
Assume that $S$ is of abelian type. Then the Andr\'e-Oort conjecture holds for $S$ and any mixed Shimura variety whose pure part is $S$.
\end{enumerate}
\end{Thm}

The conclusions of the theorem follow from Theorem~\ref{thm:main} using the Pila-Zannier strategy.
For details we refer to \cite{Gao} and \cite{Tsimerman}.

We will deduce Theorem~\ref{thm:main} from
Theorem~\ref{thm:main-counting} in the next section.  The idea
originates from a paper of the second author \cite{Schmidt}, and
applies more generally to deduce a lower bound for the degrees of
special points from the corresponding height upper bounds in a variety
of contexts (for instance, for torsion points on abelian
varieties). The same proof in fact gives a slightly more refined
statement, which we state below. For $p\in S$ a special point, let
$S(p)$ denote the smallest (zero-dimensional) special subvariety of
$S$ that contains $p$. 
It consists of the points of the image of the Shimura morphism $Sh_{K_{\bT}}(\bT, \{x\})$ in $Sh_{K}(\bG, \bX)$
induced by the inclusion of Shimura datum $(\bT, \{x\})$ in $(\bG, \bX)$,
contained in the component $S$. Since the number of components of 
$Sh_{K}(\bG, \bX)$
is independent of $p$, in all our statements and arguments, we do not 
differentiate between the cardinality of $S(p)$ and the image of  $Sh_{K_{\bT}}(\bT, \{x\})$ in $Sh_{K}(\bG, \bX)$. 
We actually prove the following.

\begin{Prop}\label{thm:height-p-implies-gal-p}
  Let $p\in S$ and let $h$ denote the maximum of $h(q)$ for
  $q\in S(p)$. Then
  \begin{equation}
    \disc(p) < C ([F(p):F]+h)^\kappa
  \end{equation}
  for some positive constants $C,\kappa$ depending only on $S$.
\end{Prop}


We describe the counting theorem which replaces the Masser-W\"ustholz isogeny estimates in
Tsimerman's strategy.  Recall (see \cite{UY3} Section 3.3 and
references therein), that $X$ is a subset of a projective variety
$\check X$ (its compact dual), naturally defined over $\ol\Q$ and that
this $\ol\Q$ structure is $\bG(\Q)$-invariant.  We can thus talk of
the set of algebraic points $X(\ol\Q)$ of $X$.  The embedding
$X \hookrightarrow \check X$ also provides us with a height function
on $X(\ol{\Q})$ and thus for any Weil height function on $S(\ol{\Q})$,
we obtain a Weil height function $h$ on $(X \times S)(\ol{\Q})$.
 
Our main technical tool is the following point-counting result. Let
\begin{equation}
  Z_S\subset X\times S, \qquad Z_S := \{(x,s) : x\in\Lambda, s=\pi(x) \}
\end{equation}
denote the graph of $\pi$ restricted to the fundamental domain $\Lambda$, and
denote
\begin{equation}
  Z_S(f,h) := \{ (x,s) \in Z_S : [F(x,s):F]\le f, h(x,s)\le h\}
\end{equation}

\begin{Thm}\label{thm:main-counting}
  We have an upper bound $\#Z_S(f,h)=\poly_S(f,h)$.
\end{Thm}

Our proof of Theorem~\ref{thm:main-counting} is based on a
polylogarithmic counting theorem \cite{Binyamini} by the first author,
sharpening Pila-Wilkie's theorem for sets defined using leaves of
foliations over number fields.

We apply this to a canonical foliation associated to the
variety $S$ to deduce Theorem~\ref{thm:main-counting}.  Note however
that the results of \cite{Binyamini} only directly apply to counting
in compact domains. We overcome this by analyzing the degeneration of
the counting constants as a function of the distance to the boundary
of the Shimura variety. The crucial input is provided by the fact that
the connection giving rise to the canonical foliation above admits
regular singularities.


Tsimerman, in \cite{Tsimerman} (Corollary 3.2), proved that for a principally polarised abelian 
variety $A$ of dimension $g$ which is simple and has CM by the ring of integers 
$O_E$, the Faltings height
$h_F(A)$ is $|\disc(E)|$-negligible.
This is a consequence of the averaged Colmez formula (see \cite{AGHM} and \cite{YZ}).

Faltings' comparison between $h_F$ and a Weil height (see \cite{Fa})
shows that our height conjecture \ref{conj:main} holds for all special
points on $\cA_g$ (for all $g>0$) corresponding to simple abelian
varieties with CM by a ring of integers of a CM field of degree $2g$
(with uniform constants). From
Proposition~\ref{thm:height-p-implies-gal-p} one then deduces that the
corresponding Galois lower bounds hold for these special points, thus
giving a new proof of Tsimerman's bound
\cite[Theorem~1.1]{Tsimerman}. This in turn implies lower bounds for
the Galois degrees of all special points of $\cA_g$ for all $g >0$
(see \cite[Theorem~5.1]{Tsimerman} and its proof) and the Andr\'e-Oort
conjecture for all Shimura varieties of abelian type.

We have allowed a slight sloppiness. The Shimura variety $\cA_g$ is not defined by a group 
of adjoint type: the centre of $GSp_{2g}$ is $\bG_m$. However, passing to the adjoint Shimura variety $\cA_g^{ad}$ does not change the discriminant of $p$ and 
the morphism $\cA_g \to \cA_g^{ad}$ is finite, hence
the height bound remains true.

The main issue is the deduction of degree (lower) bounds from the height bounds.
Tsimerman's method relies heavily on the use  of Masser-W\"ustholz isogeny estimates,
which
are only known for abelian varieties.
There is no known (even conjectural) analogue of the Masser-W\"ustholz theorem
for general Shimura varieties. 
Our method instead uses a point counting result by the first author adapted 
to the context of Shimura varieties.

\section*{Acknowledgments.}
Schmidt thanks the EPSRC for support under grant EP/N007956/1. 
Schmidt thanks the Weizmann Institute for their hospitality and the University of Basel for its support.
Yafaev is very grateful to the University of Manchester and Weizmann Institute for
  their hospitality and to Leverhulme Trust for support.  
We are very grateful to Jacob Tsimerman for pointing out a gap in the first version of the paper. 
The third author is very grateful to Emmanuel Ullmo and Rodolphe Richard for discussions related to the 
subject of this paper. The second author thanks Philipp Habegger for helpful discussions about heights. 
We are grateful to the referee for suggestions and comments.

\section{Proof of Theorem~\ref{thm:main}}

In this section we deduce Theorem~\ref{thm:main} from
Theorem~\ref{thm:main-counting} (the point-counting result). We will require a few standard
properties of special points summarised below.

In what follows we make the assumption that $K$ is neat and that $K$ is a product $K=\prod_p K_{p}$
where $K_p$ is a compact open subgroup of $\bG(\Q_p)$. This assumption does not
alter any of our bounds since replacing an arbitrary $K$ by such a subgroup only changes constants by a bounded amount
and will not affect our estimates.

We consider a special point $p$ of $S$. Let 
 $x$ be a point of $\Lambda$ such that $p= \pi(x)$.
 We let $\bT$ be the Mumford-Tate group of $x$ 
 and let $L$ be the splitting fields the splitting field of $\bT$.
 We let $d$ be the discriminant of $p$.
 
  We  let $S(p)$ be the smallest zero dimensional special subvariety of $S$
 containing $p$. 
 It is a zero dimensional Shimura variety embedded in $Sh_K(\bG,\bX)$ via the 
  inclusion of Shimura data
 $$
 (\bT,x) \subset (\bG,\bX)
 $$
All points in $S(p)$ have discriminant $d$. 
The number of points in this zero dimensional Shimura variety $S(p)$
is $\# (T(\Q) \backslash \bT(\A_f) /K_{\bT})$.

\begin{Rem}
As already mentioned in the introduction,
we have allowed a slight inaccuracy. Strictly speaking, the special subvariety
defined by the inclusion $ (\bT,x) \subset (\bG,\bX)$ may be spread among several components
of $Sh_{K}(\bG,\bX)$. Nevertheless, since we are interested in the estimates up to constants
depending only on $\bG, \bX$ and $K$, this does not change the estimates.
\end{Rem}

\begin{Prop}\label{prop:many-specials}

  The number of  points in $S(p)$ is at least
  $d^c$, for some $c=c(S)>0$.
  
  Furthermore, all points of $S(p)$ have the same degree over $F$.
\end{Prop}

\begin{proof}
The second claim follows from the definition of canonical models of zero dimensional Shimura varieties (see \cite{Milne} or \cite{EY}).
Indeed, let $r_x \colon {\rm Res}_{L/\Q}\bG_{m,L} \to T$ be the reciprocity map attached to the data $(\bT,x)$
(see \cite{EY} for example).
Let
$U$ be  $r_x(L\otimes  \A_f) \subset \bT(\A_f)$.
The fact that $\bT$ is commutative immediately shows that for any $t$ the size of the image of $U \cdot t$
in $T(\Q) \backslash \bT(\A_f) /K_{\bT}$ is the size of the image of $U$, proving the second claim.

As for the first claim, first note that
$$
\# (T(\Q) \backslash \bT(\A_f) /K_{\bT}) \gg_S [K^m_{\bT}:K_{\bT}] \times \# (T(\Q) \backslash \bT(\A_f) /K^m_{\bT})
$$
Since we have assumed $\bG$ to be adjoint,
$\bT(\R)$ is compact and therefore 
 $K^m_{\bT}\cap \bT(\Q)$ is finite.
 Its size is bounded in terms of $S$ only 
(actually, only in terms of $\dim(\bG)$) thus justifying $\gg_S$.

Again, since $\bT(\R)$ is compact, we can apply
Theorem 2.3 of \cite{UY1} (note that the same result has been obtained
independently, at the same time and by the same method by Tsimerman - see \cite{Tsimerman1}).
By that Theorem,
we have 
$$
\# (T(\Q) \backslash \bT(\A_f) /K^m_{\bT}) \gg_S d_L^{\alpha}
$$
where $\alpha$ depends on $S$ only.
The result follows with $c = \min(1,\alpha)$.
\end{proof}

\begin{Prop}\label{prop:height-bound-upstairs}
With $p$ and $x$ as above, we have
 $$
    H(x)=\poly_S(d).
$$
\end{Prop}
\begin{proof}
This is a consequence  of Theorem 1.4 (d) of \cite{DO} by Daw and Orr.
They prove that
$$
H(x) \ll_S c_1^{i(\bT)} [K^m_{\bT}:K_{\bT}]^{c_2} d_L^{c_3}
$$
where $c_1, c_2,c_3$ are constants depending on $S$ only.
The function $i(\bT)$ is the number of primes such that $K^m_{\bT,p} \not= K_{\bT,p}$
(under our assumption that $K = \prod_p K_p$, we have $K_{\bT}=\prod_p K_{\bT,p}$). 
We may assume that $c_1 > 1$ (otherwise the factor $c_1^{i(\bT)} $ disappears).
Using $2^{i(\bT)} \leq [K^m_{\bT}:K_{\bT}]$, we see that
$$
c_1^{\log([K^m_{\bT}:K_{\bT}])} \geq c_1^{i(\bT)/2}.
$$
and thus 
$$
c_1^{i(\bT)} \leq [K^m_{\bT}:K_{\bT}]^{2\log(c_1)}.
$$

We conclude that $H(x)=\poly_S(d)$.
\end{proof}

By Proposition~\ref{prop:many-specials}, there are at least $d^c$ special
points of degree $f$ over $F$.
By Proposition~\ref{prop:height-bound-upstairs} and
Conjecture~\ref{conj:main} each point $q$ in $S(p)$ gives rise to a pair
$(x_q,q)\in Z_S$ with
\begin{equation}
  h(x_q,q) = \log(\poly_S(d))+O_{S,\e}(d^\e) \qquad \text{for every }\e>0.
\end{equation}
Comparing this with Theorem~\ref{thm:main-counting} we have
\begin{equation}
  d^c < \#Z_S(f,O_{S,\e}(d^\e)) = \poly_S(f,O_{S,\e}(d^\e)).
\end{equation}
Choosing now $\e$ to be sufficiently small, we conclude that
$f>\const(S,\e')d^{c/N-\e'}$ where $N$ is the degree on the polynomial
on the right hand side and $\e'$ is any positive number. This
concludes the proof of Theorem~\ref{thm:main}. Note that the above reasoning proves 
Proposition \ref{thm:height-p-implies-gal-p}.

\section{Point counting with foliations: proof of theorem \ref{thm:main-counting}.}

In this section we recall a result from \cite{Binyamini} that will be
needed in the sequel. We state only the result that we require in the
present text, allowing us to slightly simplify the presentation. We
refer the reader to \cite{Binyamini} for some more general forms of
the counting theorem.

We use the following notation. If $D\subset\C$ (resp. $A\subset\C$) is
a disc of radius $r$ (resp. annulus of inner radius $r_1$ and outer
radius $r_2$) and $\delta\in(0,1)$, we denote by $D^\delta$ (resp
$A^\delta$) the disc (resp. annulus) with the same center and radius
$\delta^{-1}r$ (resp $\delta r_1,\delta^{-1}r_2$). We extend this
notation coordinatewise to polydiscs or polyannuli in
$\C^n$. Similarly if $B\subset\C^n$ is a ball we denote by $B^\delta$
the ball with the same center and radius $\delta^{-1}r$.

\subsection{The variety}

Let $\M\subset\A^N_\K$ be an irreducible affine variety defined over a
number field $\K$. We equip $\M$ with the standard Euclidean metric
from $\A^N$, denoted $\dist$, and denote by $\B_R\subset\M$ the
intersection of $\M$ with the ball of radius $R$ around the origin in
$\A^N$.

\begin{Rem}
  In our setting, the ambient variety will not necessarily be
  affine. It is implicitly understood that in applying
  Theorem~\ref{thm:point-counting} below, we first reduce to an affine
  cover. Note that the notion of metric $\dist$ inherited from the
  ambient affine space is not canonical, and depends on this choice of
  affine charts.
\end{Rem}

\subsection{The foliation}

Let $\vxi:=(\xi_1,\ldots,\xi_n)$ denote $n$ commuting, pointwise
linearly independent rational vector fields on $\M$ defined over $\K$.
We denote by $\cF$ the foliation of $\M$ generated by $\vxi$.

For every $p\in\M$ denote by $\cL_p$ the germ of
the leaf of $\cF$ through $p$. We have a germ of a holomorphic map
$\phi_p:(\C^n,0)\to\cL_p$ satisfying
$\partial\phi_p/\partial x_i=\xi_i$ for $i=1,\ldots,n$. We refer to
this coordinate chart as the $\vxi$-coordinates on $\cL_p$. If
$\phi_p$ continues holomorphically to a ball $B\subset\C^n$ around the
origin then we call $\cB:=\phi_p(B)$ a $\vxi$-ball. If $\phi_p$
extends to $B^\delta$ we denote $\cB^\delta:=\phi_p(B^\delta)$.

\subsection{Counting algebraic points}

Finally we are ready to state the point counting result. We fix:
$\ell\in\N$; a map $\Phi\in\cO(\M)^\ell$ defined over $\K$; and a
$\vxi$-ball $\cB\subset\B_R$ of radius at most $R$. Set
\begin{equation}
  A = A_{\Phi,\cB} := \Phi(\cB^2) \subset \C^\ell.
\end{equation}

We denote by $h:\overline\Q\to\R_{\ge0}$ the absolute logarithmic Weil
height and set
\begin{equation}
  A(g,h) := \{ p\in A : [\Q(p):\Q] \le g \text{ and } h(p)\le h\}.
\end{equation}

We deote by $\delta_\vxi$ (resp $\delta_\Phi$) the maximum of the
degree and the log-height of $\vxi$ (resp. $\Phi$). The reader may see
\cite{Binyamini} for the precise definition, or simply consider the
degrees and the heights of the coordinates of $\vxi,\Phi$ thought of
as regular function on the affine variety $\M$, i.e. as polynomials. The following is a direct consequence of \cite[Corollary~6]{Binyamini}, applied with $V=\M$.
\begin{Thm}\label{thm:point-counting}
  Suppose that for every $p\in\M$ the germ $\Phi\rest{\cL_p}$ is
  a finite map, and $\Phi(\cL_p)$ contains no germs of algebraic
  curves. Then
  \begin{equation}
    \#A(g,h) = \poly_{\M,\ell}(\delta_\vxi,\delta_\Phi,\log R,g,h).
  \end{equation}
\end{Thm}

\section{Counting special points.}

In this section, we prove Theorem~\ref{thm:main-counting} by applying the counting
results of \cite{Binyamini} to an appropriate foliation. 

Here again, we make the assumption that $K$ is neat and therefore $\Gamma = \bG(\Q)^+ \cap K$
acts without fixed points. This does not change the estimates.

The variety $S$ is equipped with a standard principal $\bG(\C)$-bundle
over $S$ given by
\begin{equation}\label{9}
  P = \Gamma\backslash (\bG(\C)\times X).
\end{equation}
We reiterate that the reference for this and the following facts is \cite[Chapter III, p.58]{Milne2} (note that our definition (\ref{9}) agrees with the Definition after Lemma 3.1 in \textit{loc. cit.} by our assumption that $G$ is adjoint).

Let us briefly recall (see \cite{Milne} and \cite{Milne2}) that each point $x\in X$ defines a Hodge character
$$
\mu_x \colon \bG_{m,\C} \longrightarrow \bG_{\C}.
$$
Fixing a faithful representation of $\bG_{\Q}$ we obtain a variation of 
rational Hodge structures over $X$ (and by passing to a quotient on $S$)
and $\mu_x$ gives rise to a Hodge filtration $F_{\mu_x}$ on the corresponding fibre of
the variation of the Hodge structure. Then $\check X$ can be identified with a
$\bG(\C)$-conjugacy class of such filtrations and the Borel embedding is
$x \mapsto F_{\mu_x}$. Recall that the projective variety $\check X$ is defined over the 
field $F$ (this is explicitly explained in \cite{UY3} , section 3.3).

Reverting to the $\bG(\C)$-bundle $P$, the corresponding flat
$\bG(\C)$-structure corresponds to a flat, regular-singular 
$\bG(\C)$-connection on $P$, which is also defined over $F$ (this
follows from Theorem 5.1 of \cite{Milne2}). There is a
$\bG(\C)$-equivariant map
\begin{equation}
  \beta:P\to \check X, \qquad \beta([g,x]) = g^{-1} F_{\mu_x}.
\end{equation}
Denote by $\pi_P:P\to S$ the projection map. Then
\begin{equation}
  \pi_P([1,x]) = \pi\circ\beta([1,x]).
\end{equation}

There exists a finite collection of affine opens $\{U_i\subset S\}$ and
trivialising charts $P\rest {U_i}\simeq U_i\times\bG$. In this
trivialisation the canonical connection $\nabla$ takes the form
\begin{equation}
  \nabla(\sigma) = \d \sigma - \Omega_i\cdot\sigma,\qquad \Omega_i\in\fg(\Omega^1(U_i))
\end{equation}
where $\sigma$ is a section of $P\rest{U_i}$, $\fg$ is the complex Lie
algebra associated to $\bG$, and $\Omega^1(U_i)$ denotes the space of
algebraic one-forms on $U_i$. Recall that $\nabla$, and hence
$\Omega_i$, are defined over $F$.

By Hironaka's desingularisation theorem we may fix a projective
variety $\bar U_i$, also defined over $F$, such that
$U_i\subset\bar U_i$ is dense and $\bar U_i\setminus U_i$ consists of
normal crossings divisors. That is, around every point $s\in\bar U_i$
there exists a system of parameters $(x_{i1},\ldots,x_{in})$ such that
$\bar U_i\setminus U_i$ is given by the zero locus of some monomial in
these parameters, which we assume for simplicity to be
\begin{equation}
  \bar U_i\setminus U_i = \{x_{i,1} \ldots x_{i,k_s}=0\}, \qquad k_s\in\{0,\ldots,n\}.
\end{equation}
By compactness of each $\bar U_i$ in the complex topology, we may
cover  $\overline{S}$ (see Remark 3) by a finite set of complex polydiscs $B_\alpha$ with a
system of parameters $(x_{\alpha,1},\ldots,x_{\alpha,n})$, defined
over $F$, and
\begin{equation}
  B_\alpha^\circ := S\cap B_\alpha = B_\alpha\cap \{x_{\alpha,1}\cdots x_{\alpha,k_\alpha}\neq0\}.
\end{equation}
We may assume that $(B_\alpha^\circ)^{1/4}\subset U_i$. Moreover we
may after rescaling assume that each $B_\alpha$ is the unit
polydisc. We denote the number of polydiscs by $N$. We denote by 
\begin{equation*}
	S(f,h) = \{s \in S(\mathbb{C}) ; \ h(s) \leq h, [F(s):F] \leq f\}. 
\end{equation*}

\begin{Lem}\label{lem:orbit-polyannulus}
  A fraction of at least $1/(2N)$ of the points in $S(f,h)$ lie in the
  polyannulus
  \begin{equation}
    A_\e := \{ \e < |x| < 1 \}^{\times k_\alpha}\times\{|x|<1\}^{\times n-k_\alpha}, \qquad \e=e^{-O_S(Nh)}
   \end{equation}
  inside one of the polydiscs $B_\alpha$ above.
\end{Lem} 
We note that the constant in $O_S(Nh)$ above also depends on the affine cover and the choice of polydiscs but we may consider it fixed. 
\begin{proof}
  Let $s\in S(f,h)$, so that $h(s)\le h$ and $[F(s):F]\le f$. Note
  that $S(f,h)$ is invariant under the action of the Galois group
  $\Gal(\bar F/F)$. Denote by $\cO_s\subset S(f,h)$ the Galois orbit
  of $s$.  Since the coordinates $x_{\alpha,i}$ are defined over $F$, it holds for the naive Weil-height $h(x_{\alpha,i})$ that
  \begin{equation*} 
  	h(x_{\alpha,i}(s))=O_S(h), \ \alpha = 1, \dots, N, i =1,\dots, n. 
  	\end{equation*} 
  A fraction of $1/N$ of the conjugates of $s$ belong to a
  single $B_\alpha$ -- denote these by $\cO_{s,\alpha}$.  Without loss of generality we may assume that $s\in B_\alpha$ (otherwise replace
  $s$ by a conjugate).
 As $h(x) = h(1/x)$ for $x \neq 0$ it follows from the definition of the Weil-height as a sum of local heights  \cite[1.5.7]{BG} that 
  \begin{equation*}
\sum_{s_\sigma\in\cO_{s,\alpha}} -\log |x_{\alpha,i}(s_\sigma)| = O_S(   \#\cO_s h) , i = 1, \dots, k_\alpha.  
  \end{equation*}  
At this point we have used that $B_{\alpha}$ is a polydisc of radius 1. As  $\#\cO_{s,\alpha} \geq \#\cO_{s}/N$ we also have $O_S( \#\cO_s h) = O_S(N \# \cO_{s,\alpha}h)$. 
  In particular at least $1-1/(2k_\alpha)$ of the points in
  $\cO_{s,\alpha}$ satisfy $-\log |x_{\alpha,i}(s_\sigma)|=O_S(2k_\alpha Nh) = O_S(Nh)$. Repeating this for
  $i=1,\ldots,k_\alpha$, we see that half the points in
  $\cO_{s,\alpha}$ satisfy this estimate for all these coordinates
  concurrently. This proves the claim.
\end{proof}

According to Lemma~\ref{lem:orbit-polyannulus} it will suffice to
count the $(f,h)$-points in the polyannulus $A_\e$ inside each
polydisc $B_\alpha$. So below we fix one $U=U_i$ and one
$B=B_\alpha\subset U$, with the correponding coordinate system
$(x_1,\ldots,x_n)$ and $k=k_\alpha$. We also write $\Omega=\Omega_i$.

We consider the space $\M=P\rest U=U\times\bG(\C)$ with its foliation $\cF$
by horizontal sections of $\nabla$. Explicitly, this is the foliation
generated by the vector fields $\xi_1,\ldots,\xi_n$ where
\begin{equation}
  \xi_j = \pd{}{x_j} + \Omega(\pd{}{x_j})\cdot g
\end{equation}
for $(x,g)\in U\times\bG(\C)$. The leaves of $\cF$ are given by
$\{\cL_g\}_{g\in\bG(\C)}$ where
\begin{equation}
  \cL_g = \{[x,g] : x\in X\} \subset P.
\end{equation}
In particular set $\cL:=\cL_1$, and recall that $\cL_g=g\cdot\cL_1$.

Consider the map
\begin{equation}
  \Phi:P\to \check X\times S, \qquad \Phi=(\beta,\pi_P),
\end{equation}
and note that $\Phi(\cL_1)$ is the graph of $\pi$, and in particular
$\Phi(\cL_1)\cap(\Lambda\times S)=Z_S$ (where we identify $\Lambda$
with its image in $\check X$ under the map $X\to\check X$).

\begin{Lem}
  The fibers of $\Phi$ are zero-dimensional on every leaf $\cL_g$, and
  $\Phi(\cL_g)$ contains no germs of algebraic curves.
\end{Lem}
\begin{proof}
  Since $\cL_g=g\cL_1$ and $\beta$ is $\bG(\C)$-equivariant, it is
  enough to prove the claim for $\cL_1$. The first claim is obvious,
  since both $\beta$ and $\pi_P$ each form a local system of
  coordinates at every point of $\cL_1$. Since $\Phi(\cL_1)$ is the
  graph of $\pi$, the second claim follows from general functional
  transcendence statements for Shimura varieties. However, we give a
  completely elementary argument below.

  Suppose that the graph $G_\pi$ of $\pi:X\to S$ contains the germ of an
  irreducible algebraic curve $C\subset\check X\times S$ at a point
  $p\in\check X\times S$. We will show that in this case
  $C\subset X\times S$. Since $X$ is a bounded symmetric domain, this
  implies that $C$ projects to a point in $\check X$, contradicting
  the fact that the germ of $C$ at $p$ is a subset of $G_\pi$.

  To prove the claim let $q=(q_X,q_S)\in C$ and we will prove
  $q_X\in X$. Recall that $C$, as an irreducible curve, is pathwise
  connected, and let $\gamma=(\gamma_X,\gamma_S):[0,1]\to C$ be a path
  with $\gamma(0)=p$ and $\gamma(1)=q$. Let $\gamma'_X:[0,1]\to X$
  denote the unique lifting of the path $\gamma_S$ along $\pi$ with
  $\gamma'(0)=p_X$. We will show that $\gamma_X\equiv \gamma'_X$, and
  in particular $q_X=\gamma_X(1)\in X$ as claimed. To see this, let
  $t_0$ denote the maximal $t$ such that for every $t\le t_0$ we have
  $\gamma_X(t)=\gamma'_X(t)$, and suppose toward contradiction that
  $t_0<1$. The path $\gamma\rest{[0,t_0]}$ belongs $C\cap G_\pi$, and by
  analytically continuing along this path we see that the germ of $C$ at
  $\gamma(t_0)$ belongs to $G_\pi$ as well.  In particular
  $\pi(\gamma_X(t))=\gamma_S(t)$ for every $t$ in some small
  neighborhood of $t_0$. Thus $\gamma_X$ remains the (unique) lifting
  of $\gamma_S$, and agrees with $\gamma'_X$, in this neighborhood of
  $t_0$ -- contradicting the maximality of $t_0$.

\end{proof}

We will obtain an upper bound for $\#Z_S(f,h)$ by applying
Theorem~\ref{thm:point-counting} to the foliation $\cF$, the leaf
$\cL$, and the map $\Phi$ constructed above. However, since
Theorem~\ref{thm:point-counting} only directly applies to $\xi$-balls,
a further covering argument will be needed. We will require the
following growth estimate for the leaf $\cL$. Let
$\hat B^\circ\subset B$ denote the domain
\begin{equation}
  \hat B^\circ = B \setminus \bigcup_{i=1}^k \{x_i\in(-\infty,0]\}.
\end{equation}
(Note that $\hat B^\circ \subset B^\circ$). Fix some basepoint $(s_0,g_0)\in\cL$ with $s_0\in \hat B^\circ$, and
let $g:\hat B^\circ\to G$ be the multivalued function given by
analytically continuing from $g(s_0)=g_0$ as a flat section of
$\nabla$. This corresponds to a subset of $\cL$ that we denote
$\cL_{B^\circ,s_0,g_0}$.
  
\begin{Lem}\label{lem:leaf-growth}
  Consider $g$ constructed above and
  $s\in A^{1/2}_\e\cap \hat B^\circ$. Then
  \begin{equation}
    \log \norm{g(s)} = \poly_S(|\log \e|) = \poly_S(h)
  \end{equation}
\end{Lem}
\begin{proof}
  This essentially follows by the regularity of the canonical
  connection $\nabla$. More explicitly, since $\pi_1(B^\circ)$ is
  commutative, it follows that the monodromy operators $M_i$ along
  $x_i=0$, for $i=1,\ldots,k$, commute. Let $L_i$ be such that 
\begin{equation*}
\exp(2\pi i L_i) = M_i^{-1}
\end{equation*}
and such that $L_1, \dots, L_n$ (pairwise) commute (that this choice is possible follows from \cite[Lemma IV.4.5]{Borel}). Then
  \begin{equation}
    \hat g = x_1^{L_1} \cdots x_k^{L_k} g
  \end{equation}
  is a univalued matrix function. As a flat section of the regular
  connection $\nabla$, $g$ admits regular growth along any analytic
  curve in $B$. The same is then obviously true for $\hat g$, and we
  conclude that $\hat g$ is meromorphic in $B$, with poles along
  $x_i=0$ for $i=1,\ldots,k$. The estimate for $\hat g(s)$ follows by
  standard theory of meromorphic functions. A similar estimate for
  $x_i^{L_i}$ follows immediately by direct computation.
\end{proof}

Now we return to the proof of Theorem \ref{thm:main-counting}. Recall that the map $\pi:X\to S$ restricted to $\Lambda$ is definable
in the  o-minimal structure $\R_{an,exp}$ (see \cite{KUY}, Theorem 4.1). It follows
that $\Lambda\cap\pi^{-1}(\hat B^\circ)$ has finitely many connected
components (with their number depending on $S$). The part of $Z_S$
lying over $\hat B^\circ$ is therefore contained in the union of
finitely many sets of the form $\Phi(\cL_{B^\circ,s_0,g_0})$. It will
suffice to count the algebraic points in each of these sets
separately. Denote one such set by $Z'$ below.

To finish the counting, set $\e=e^{-O_S(h)}$ (with $O_S(h) = O_S(Nh)$ but we have fixed $N$) as in
Lemma~\ref{lem:orbit-polyannulus}. Up to the constant factor $1/(2N)$,
it will suffice to count the points lying over $A_\e$. Cover $A_\e$ by
polydiscs $B_j$, with $B_j^{1/2}\subset A_\e^{1/2}$ and
\begin{equation}
  \#\{B_j\} = \poly_S(|\log\e|)=\poly_S(h).
\end{equation}
This can be achieved by a simple logarithmic subdivision
process. Namely, for each $i=1,\ldots,k$ we use $O(1)$ discs to cover
$\{1/2<|x_j|<1\}$, then $O(1)$-discs to cover $\{1/4<|x_j|<1/2\}$ and
so on. Taking direct products of the collections obtained for each
coordinate (and the unit disc for $x_{i+1},\ldots,x_n$) gives the
required collection.

Let $\cB_j$ be the $\xi$-ball given in the $x$-coordinates by $B_j$,
and in the $g$-coordinates by analytically continuing from $s_0$ to
the $B_j$ along a path inside $\hat B^\circ$. In light of
Lemma~\ref{lem:leaf-growth} each $\cB_j$ is contained in a ball of
radius $R$ in $\M$, with $\log R=\poly_S(h)$. By construction the
union of the images of images of $\Phi(\cB_j)$ covers $Z'$. Finally,
applying Theorem~\ref{thm:point-counting} to each $\cB_j$ finishes the
proof.

\textit{On behalf of all authors, the corresponding author states that there is no conflict of interest.}

\end{document}